\input ./style/arxiv-vmsta.cfg
\documentclass[numbers,compress,v1.0.1]{vmsta}
\usepackage{vtexbibtags}

\volume{7}% Updated by VTEXPTS2LaTeX.exe, 25.03.2020 14:35
\issue{1}% Updated by VTEXPTS2LaTeX.exe, 25.03.2020 14:35
\pubyear{2020}
\firstpage{97}% Updated by VTEXPTS2LaTeX.exe, 25.03.2020 14:35
\lastpage{112}% Updated by VTEXPTS2LaTeX.exe, 25.03.2020 14:35
\aid{VMSTA151}% Updated by VTEXPTS2LaTeX.exe, 13.03.2020 15:13
\doi{10.15559/20-VMSTA151}% Updated by VTEXPTS2LaTeX.exe, 13.03.2020 15:13
\articletype{research-article}

\startlocaldefs
\newtheorem{theorem}{Theorem}
\newtheorem{proposition}{Proposition}
\newtheorem{corollary}[theorem]{Corollary}
\theoremstyle{definition}
\newtheorem{remark}{Remark}[section]

\allowdisplaybreaks

\numberwithin{equation}{section}

\endlocaldefs

\begin{document}

\begin{frontmatter}
\pretitle{Research Article}

\title{Alternative probabilistic representations of Barenblatt-type solutions}

\author{\inits{A.}\fnms{Alessandro}~\snm{De Gregorio}\thanksref{cor1}\ead[label=e1]{alessandro.degregorio@uniroma1.it}\orcid{0000-0003-0809-6046}}
\author{\inits{R.}\fnms{Roberto}~\snm{Garra}\ead[label=e2]{roberto.garra@uniroma1.it}}
\thankstext[type=corresp,id=cor1]{Corresponding author.}
\address{Department of Statistical Sciences, \institution{``Sapienza'' University of Rome},\break  P. le Aldo Moro, 5 -- 00185, Rome, \cny{Italy}}

%\thankstext[id=f1]{}

%\dedicated{}

%\markboth{Authors}{Title}
\markboth{A. De Gregorio, R. Garra}{Alternative probabilistic representations of Barenblatt-type solutions}

\begin{abstract}
%In this paper we consider
A general class of probability density functions
\[
u(x,t)=C t^{-\alpha d}\left (1-\left (\frac{\|x\|}{ct^{\alpha }}
\right )^{\beta }\right )_{+}^{\gamma },\quad x\in \mathbb{R}^{d},t>0,
\]
is considered, containing as particular case the Barenblatt solutions arising, for instance,
in the study of nonlinear heat equations.
%We propose
Alternative probabilistic
representations of the Barenblatt-type solutions $u(x,t)$ are proposed. In the one-dimensional
case, by means of this approach,
%we are able to connect
$u(x,t)$ can be connected with the
wave propagation.
\end{abstract}

\received{\sday{29} \smonth{11} \syear{2019}}% Updated by VTEXPTS2LaTeX.exe, 13.03.2020 15:13
\revised{\sday{4} \smonth{3} \syear{2020}}% Updated by VTEXPTS2LaTeX.exe, 13.03.2020 15:13
\accepted{\sday{13} \smonth{3} \syear{2020}}% Updated by VTEXPTS2LaTeX.exe, 13.03.2020 15:13
\publishedonline{\sday{23} \smonth{3} \syear{2020}}

\end{frontmatter}
%s1 #&#
\section{Introduction}
\label{sec1}

The Brownian motion\index{Brownian motion} is a stochastic process of interest for pure
and applied mathematicians. The probability distribution of the Brownian
motion\index{Brownian motion} is given by the Gaussian kernel
\begin{equation*}
G(x,t)=\frac{1}{(4\pi t)^{d/2}}\exp (-\|x\|^{2}/4t), \quad x\in
\mathbb{R}^{d}, t>0,
\end{equation*}
which is the source-type solution (that is $G(x,0)=\delta (x)$, where
$\delta (x)$ represents Dirac's delta function) to the parabolic heat
equation\index{parabolic heat equation}
\begin{equation*}
\frac{\partial u}{\partial t}=\Delta u.
\end{equation*}
The main %shortcoming of the previous
unwanted feature of this solution is that
%is automatically
it is inevitably
positive everywhere in its domain of definition; i.e., the Brownian motion\index{Brownian motion}
scatters with unbounded velocity. A way to overcome this %drawback
feature is to
consider the porous medium equation\index{Porous Medium Equation (PME)} which is a nonlinear diffusion equation
%e1 #&#
\begin{equation}
\label{eq:pme0}
\frac{\partial u}{\partial t}=\Delta (u^{m}),\quad m>1,
\end{equation}
having the source-type solution given by
\begin{equation*}
U(x,t)=Ct^{-\alpha d}\left (1-\frac{\|x\|^{2}}{c^{2}t^{2\alpha }}
\right )_{+}^{\frac{1}{m-1}},\quad \alpha >0,c>0,
\end{equation*}
where $(x)_{+}:=\max (x,0)$ and $C$ is a suitable constant such that
$\int U(x,t) \mathrm{d}x=1$. The solution $U(\cdot ,t)$ is a compactly
supported function and it is called the  \emph{Barenblatt solution}. For a complete
description of the mathematical analysis %theory
related to the partial
differential equation \eqref{eq:pme0} the reader can consult
\cite{vazquez}. The connection between the porous medium equation\index{Porous Medium Equation (PME)} and the
theory of stochastic processes has been investigated, for instance, in
\cite{inoue,inoue2,inoue3,ben,feng,jou,phil} and \cite{dgo2}.

The aim of this paper is to study a class of functions generalizing the
Barenblatt solution $U(x,t)$.\index{Barenblatt solution} For fixed $t>0$, we %deal with
consider the map
%e2 #&#
\begin{align}
\label{eq:gen}
\mathbb{R}^{d} \ni x\mapsto u:=u(x,t):=C t^{-\alpha d}\left (1-\left (
\frac{\|x\|}{ct^{\alpha }}\right )^{\beta }\right )_{+}^{\gamma },
\end{align}
where $\alpha >0,\beta >0$ and $\gamma >0$, and
\begin{equation*}
C:=C(\beta ,\gamma ,d):=
\frac{\beta }{c^{d}\sigma (\mathbb{S}^{d-1})\mathrm{Beta}(\frac{d}{\beta },\gamma +1)}
\end{equation*}
is a positive constant determined by the condition
$\|u(x,t)\|_{L^{1}(\mathbb{R}^{d},\mathrm{d}x)}=1$ (the property of the mass conservation), $\mathrm{d}x$ denoting the Lebesgue measure on
$(\mathbb{R}^{d},\mathcal{B}(\mathbb{R}^{d}))$,
$\sigma (\mathbb{S}^{d-1}):=2\pi ^{d/2}/\Gamma (d/2)$ represents the surface
area of the $(d-1)$-dimensional sphere $\mathbb{S}^{d-1}$ with radius one
and $\mathrm{Beta}(a,b):=\frac{\Gamma (a)\Gamma (b)}{\Gamma (a+b)}, a,b>0$. We observe
that $u$ is a probability density function\index{probability density function} with an associated absolutely continuous
probability measure given by
%e3 #&#
\begin{equation}
\mu _{t}(\mathrm{d}x)=u(x,t)\mathrm{d}x,
\end{equation}
having the following features:
\begin{itemize}
\item The density function $u$\index{density function} is compactly supported; i.e., for every
$t>0$, the support of $u(\cdot ,t)$ is given by
\begin{equation*}
\text{supp}\, u(\cdot ,t)= B_{r(t)}:=\{x\in \mathbb{R}^{d}:\|x\|\leq r(t)
\},
\end{equation*}
representing a closed ball with radius $r(t):=ct^{\alpha }$. This property
implies the finite speed of propagation of $u$; i.e., the $(d-1)$-dimensional
sphere with radius $r(t)$ denoted by $\mathbb{S}^{d-1}_{r(t)}$ provides
the free boundary separating the regions
$\{(x,t)\in \mathbb{R}^{d} \times (0,\infty ):u(x,t)>0\}$ and
$\{(x,t)\in \mathbb{R}^{d} \times (0,\infty ):u(x,t)=0\}$.
\item The probability measure $\mu _{t}$ is rotationally invariant; that
is, $\mu _{t}(\mathrm{d}x)=\break u(\|x\|,t)\mathrm{d}x$, or equivalently,
$\mu (M A)=\mu (A)$ for $A\in \mathcal{B}(\mathbb{R}^{d})$ and
$M\in O(d)$, where $O(d)$ is the group of $d\times d$ orthogonal matrices
acting in $\mathbb{R}^{d}$, where $d\geq 2$. As a direct consequence of this
property we derive
%e4 #&#
\begin{align}
\label{eq: ball}
\mu _{t}(B_{a})&=\sigma (\mathbb{S}^{d-1})\int _{0}^{a} r^{d-1}u(r,t)
\mathrm{d}r\nonumber\\
&=\frac{\mathrm{Beta}\left ((a/ct^{\alpha })^{\beta };\frac{d}{\beta },\gamma +1\right )}{\mathrm{Beta}(\frac{d}{\beta },\gamma +1)},
\quad 0<a<ct^{\alpha },
\end{align}
where $\mathrm{Beta}(x;a,b)=\int _{0}^{x} y^{a-1}(1-y)^{b-1}\mathrm{d}y, x
\in \mathbb{R}$, is the incomplete Gamma function.
\item The density $u$ is a self-similar function. Indeed
\begin{equation*}
u(x,t)=L^{d\alpha }u(L^{\alpha }x, Lt), \quad L>0.
\end{equation*}
\end{itemize}

We refer to \eqref{eq:gen} as the class of the \textit{Barenblatt-type solutions}.
The family of functions \eqref{eq:gen} contains as particular cases, for
instance, weak solutions of several nonlinear and linear diffusion equations
(see Section \ref{sec2}). Furthermore, in \cite{getoor} it is proved that
the mean exit time of a symmetric L\'{e}vy stable process\index{symmetric L\'{e}vy stable process} from a ball admits
a representation belonging to the Barenblatt-type solution class. In this
paper we provide alternative probabilistic representations of Barenblatt-type
density functions in terms of mean value of delta functions containing
random terms (see Section \ref{sec3}). At least in the case $d=1$, our
approach permits to shed light on the connection of the nonlinear diffusion
with the propagation of waves and spherical waves (which are described
by means of linear partial differential equations). The main novelty of
this interpretation is that a wave performs random displacements nonlinearly
with respect to time. It is worth to mention that solutions belonging
to family \eqref{eq:gen} emerge from different frameworks (linear hyperbolic,\index{linear hyperbolic}
nonlinear parabolic\index{nonlinear parabolic} and nonlocal\index{nonlocal}). %Therefore,
In this way, objects requiring different
mathematical tools have %joint
common features.

It is worth to %mention
note that the connections between stochastic processes
and nonlinear Fokker--Planck equations have also been analyzed in
\cite{leo,mendes} and in references therein.

%s2 #&#
\section{Barenblatt-type solutions to diffusion equations}
\label{sec2}

The aim of this section is to highlight that the class of density functions\index{density function}
of the form \eqref{eq:gen} is very general. Solutions belonging to family
\eqref{eq:gen} appear in different frameworks (linear hyperbolic,\index{linear hyperbolic} nonlinear
parabolic\index{nonlinear parabolic} and nonlocal\index{nonlocal}). Therefore, in what follows, we list some diffusion
equations studied by means of different approaches. Nevertheless,
%they share the same analytic form of
their solutions share the same analytic form.

The Fourier transform $\mathcal{F}$ and the inverse transform
$\mathcal{F}^{-1}$ of a function
$v\in\break L^{1}(\mathbb{R}^{d}, \mathrm{d}x)$ are defined by
\begin{equation*}
\mathcal{F} v(\xi )=\int _{\mathbb{R}^{d}} v(x) e^{i x\cdot \xi }
\mathrm{d}x,\qquad \mathcal{F}^{-1} v(x)=\frac{1}{(2\pi )^{d}}\int _{
\mathbb{R}^{d}} v(\xi ) e^{-i x\cdot \xi }\mathrm{d}\xi ,
\end{equation*}
with $\xi \in \mathbb{R}^{d}$.

%s2.1 #&#
\subsection{Nonlinear diffusions:\index{nonlinear diffusions} $p$-Laplacian equation}
\label{sec2.1}

We %deal with
mean here the $p$-Laplacian equation (PLE for short) studied, for instance,
in \cite{kam} and \cite{lee}, %given by
which is the following nonlinear degenerate
parabolic evolution equation
%e5 #&#
\begin{equation}
\label{eq:plaplacian}
\frac{\partial u}{\partial t}=\text{div}\left (|\nabla u|^{p-2}\nabla u
\right ),\quad p>2,\quad  t>0,
\end{equation}
subject to the initial condition
%e6 #&#
\begin{equation}
\label{initcon0}
u(x,0)=\delta (x),
\end{equation}
where $u:=u(x,t)$, with $x\in \mathbb{R}^{d},d\geq 1$, is a scalar function
defined on $\mathbb{R}^{d}\times \mathbb{R}^{+}$. The Cauchy problem
\eqref{eq:plaplacian}--\eqref{initcon0} admits a unique nonnegative fundamental
solution:\index{fundamental solution} it is  a function $u\geq 0$ solving \eqref{eq:plaplacian}--\eqref{initcon}
in a weak sense (see \cite{kam} and \cite{lee} for the detailed definition
of weak solution to PLE). This solution is given by the following probability
density function\index{probability density function}
%e7 #&#
\begin{align}
\label{eq:weaksol}
u(x,t)= t^{-k}\left (\mathfrak{c}-q\left (\frac{\|x\|}{t^{k/d}}\right )^{
\frac{p}{p-1}}\right )_{+}^{\frac{p-1}{p-2}},
\end{align}
where
\begin{equation*}
k:=\left (p-2+\frac{p}{d}\right )^{-1},\qquad q:=\frac{p-2}{p}\left (
\frac{k}{d}\right )^{\frac{1}{p-1}}
\end{equation*}
and $\mathfrak{c}:=\mathfrak{c}(p,d)$ is a constant determined by the condition
$\int u(x,t)\mathrm{d}x=1$. By setting
$\beta =\frac{p}{p-1}, \gamma =\frac{p-1}{p-2}, \alpha =\frac{k}{d}, C=
\mathfrak{c}^{\frac{p-1}{p-2}}$ and
$c=(\mathfrak{c}/q)^{\frac{p-1}{p}}$, the Barenblatt-type solution
\eqref{eq:gen} coincides with \eqref{eq:weaksol}.

%s2.2 #&#
\subsection{Nonlinear diffusions:\index{nonlinear diffusions} nonlocal porous medium equation}
\label{sec2.2}

The Nonlocal Porous Medium Equation\index{Nonlocal Porous Medium Equation (NPME)} (NPME), studied in \cite{biler0} and
\cite{biler}, is %given by
the following degenerate nonlinear and nonlocal\index{nonlocal}
evolution equation
%e8 #&#
\begin{equation}
\label{eq:nonloceq}
\frac{\partial u}{\partial t}=\text{div}\left (|u|\nabla ^{\nu -1}(|u|^{m-2}u)
\right ),\quad m>1,\,\nu \in (0,2],\,t>0,
\end{equation}
subject to the initial condition
%e9 #&#
\begin{equation}
\label{initcon}
u(x,0)=u_{0}(x).
\end{equation}
The pseudo-differential operator $\nabla ^{\nu -1}$ is the fractional gradient
denoting the nonlocal operator\index{nonlocal operator} defined as
$\nabla ^{\nu -1}u:=\mathcal{F}^{-1}(i\xi \|\xi \|^{\nu -2}\mathcal{F}u)$.
This notation highlights that $\nabla ^{\nu -1}$ is a pseudo-differential
(vector-valued) operator of order $\nu -1$. Equivalently, we can define
$\nabla ^{\nu -1}$ as $\nabla (-\Delta )^{\frac{\nu }{2}-1}$, where
$(-\Delta )^{\frac{\nu }{2}}u=\mathcal{F}^{-1}(\|\xi \|^{\nu }\mathcal{F}u)$ is
the fractional Laplace operator,\index{fractional Laplace operator} i.e., a Fourier multiplier
%who symbol is given by
with the symbol $\|\xi \|^{\nu }$. For $\nu =2$, \eqref{eq:nonloceq} becomes the
classical nonlinear porous medium equation\index{Porous Medium Equation (PME)}
%e10 #&#
\begin{equation}
\label{eq:pme}
\frac{\partial u}{\partial t}=\text{div}\left (|u|\nabla (|u|^{m-2}u)
\right )=\text{div}\left ((m-1)|u|^{m-1}\nabla u\right ).
\end{equation}
If we restrict our attention to nonnegative solution $u(x,t)$, equation
\eqref{eq:pme} becomes
%e11 #&#
\begin{equation}
\label{eq:mpme}
\frac{\partial u}{\partial t}=\frac{m-1}{m}\Delta (u^{m}),
\end{equation}
which is usually adopted to model the flow of a gas through a porous medium.

Let $\nu \in (0,2]$ and $m>1$. A weak solution, in the sense of Definition
1 in \cite{biler}, is given by
%e12 #&#
\begin{align}
\label{eq:weaksol2}
u(x,t)= Ct^{-d\alpha }\left (1-k^{\frac{2}{\nu }}
\frac{\|x\|^{2}}{t^{2\alpha }}\right )_{+}^{\frac{\nu }{2(m-1)}},
\end{align}
where $\alpha :=\frac{1}{d(m-1)+\nu }$,
$k:=
\frac{d\Gamma (d/2)}{(d(m-1)+\nu )2^{\nu }\Gamma (1+\frac{\nu }{2})\Gamma (\frac{d+\nu }{2})}$
and
\begin{equation*}
\quad C:=
\frac{\Gamma (\frac{d}{2}+\frac{\nu }{2(m-1)}+1)k^{\frac{ d}{\nu }}}{\pi ^{\frac{d}{2}}\Gamma (\frac{\nu }{2(m-1)}+1)}.
\end{equation*}
Furthermore, $u(x,t)$ is the pointwise solution of equation
\eqref{eq:nonloceq} for $\|x\|\neq k^{-\frac{1}{\nu }}t^{\alpha }$. The link
between \eqref{eq:weaksol2} and random flights has been investigated in
\cite{DG18}. For $\nu =2$, the solution \eqref{eq:weaksol} becomes the
Barenblatt--Kompanets--Zel'dovich--Pattle solution of the porous medium equation\index{Porous Medium Equation (PME)}
\eqref{eq:mpme} supplemented with the initial condition
$u(x,0)=\delta (x)$ (see, for instance, \cite{vazquez}).

The function \eqref{eq:gen} reduces to \eqref{eq:weaksol2} for
$\gamma =\frac{\nu }{2(m-1)}, \beta =2$ and $c=1/k^{\frac{2}{\nu }}$.

%s2.3 #&#
\subsection{Euler--Poisson--Darboux equation}
\label{sec2.3}

It is well known that the fundamental solution\index{fundamental solution} of the Euler--Poisson--Darboux
(EPD) equation
%e13 #&#
\begin{equation}
\frac{\partial ^{2} u}{\partial t^{2}}+\frac{d+2\nu -1}{t}
\frac{\partial u}{\partial t} = c^{2} \Delta u, \quad \nu >0, \ t>0,
\ c>0,
\end{equation}
has the form
%e14 #&#
\begin{equation}
\label{epd}
u(x,t) = \frac{\Gamma (\nu +\frac{d}{2})}{\pi ^{d/2}\Gamma (\nu )}
\frac{1}{(ct)^{d}}\left (1-\frac{\|x\|^{2}}{(ct)^{2}}\right )_{+}^{
\nu -1}
\end{equation}
and therefore it belongs to the family of probability density functions
with compact support \eqref{eq:gen} with $\beta = 2$, $\alpha = 1$,
$k = d$ and $\gamma = \nu -1$. There is a wide literature about the EPD
equation\index{EPD equation} and its applications. We refer to Bresters \cite{br} for the construction
of weak solutions of the initial value problem for the EPD equation\index{EPD equation} based
on distributional methods. We recall that, in the one-dimensional case,
the solution to the Cauchy problem
%e15 #&#
\begin{equation}
\label{1.1}
\begin{cases}
\displaystyle\frac{\partial ^{2} u}{\partial t^{2}}+\frac{2\xi }{t}
\frac{\partial u}{\partial t}= c^{2}
\frac{\partial ^{2} u}{\partial x^{2}},
\\
\displaystyle u(x,0)=f(x),
\\
\displaystyle\frac{\partial u}{\partial t}(x,t)\big |_{t=0}=0,
\end{cases}
\end{equation}
can be represented as the Erd\'{e}lyi--Kober fractional integral (see
definition \eqref{eq:eki} below) of the D'Alembert solution of the wave
equation (see \cite{erd})
%e16 #&#
\begin{equation}
\begin{cases}
\displaystyle\frac{\partial ^{2} w}{\partial t^{2}}= c^{2}
\frac{\partial ^{2} w}{\partial x^{2}},
\\
\displaystyle w(x,0)= f(x),
\\
\displaystyle\frac{\partial w}{\partial t}(x,t)\big |_{t=0}=0.
\end{cases}
\end{equation}

This means that
%e17 #&#
\begin{align}
u(x,t)&=\frac{2}{\mathrm{Beta}(\xi , \frac{1}{2})}\int _{0}^{1} (1-y^{2})^{
\xi -1}\left [\frac{f(x+yct)+f(x-yct)}{2}\right ]\mathrm{d}y.
\label{2}
\end{align}
The first probabilistic interpretation\index{probabilistic interpretation} of this analytic representation, discussed
by Rosencrans \cite{Rosen} and more recently by Garra and Orsingher
\cite{garra3}, is the following one: solution \eqref{2} can be written
as
%e18 #&#
\begin{equation}
u(x,t) = \mathbf{E}\left [
\frac{f(x+\,\mathcal{U}(t))+f(x-\,\mathcal{U}(t))}{2}\right ],
\end{equation}
where
%e19 #&#
\begin{equation}
\label{agost}
\mathcal{U}(t) = U(0)\int _{0}^{t}(-1)^{N(s)}\mathrm{d}s,
\end{equation}
and $(N(t))_{t\geq 0}$ is the nonhomogeneous Poisson process with rate
$\lambda (t)= \frac{\xi }{t}$,  $U(0)$ is a uniformly distributed r.v.
on $\{-c,c\}$ (furthermore $(N(t))_{t\geq 0}$ and $U(0)$ are supposed independent).
By means of the general Proposition~\ref{prop1} (see the next section), we here obtain
a new interesting probabilistic interpretation\index{probabilistic interpretation} of the fundamental solution\index{fundamental solution}
of the EPD equation.\index{EPD equation}

Moreover, we have the following interesting picture that underlines the
role of the Barenblatt-type solution as a bridge between nonlinear and linear
PDEs:
\begin{itemize}
\item The Erd\'{e}lyi--Kober fractional integral of the solution of the D'Alembert
equation leads to the solution of the Euler--Poisson--Darboux equation under
the same initial conditions.
\item As recently pointed out in \cite{dgo2} and \cite{dgo}, the time-rescaled
Kompanets--\xch{Zel'do\-vich}{Zelddovich}--Barenblatt solution of the Porous Medium Equation\index{Porous Medium Equation (PME)}
(PME) coincides with the fundamental solution\index{fundamental solution} of the Euler--Poisson--Darboux
equation.
\end{itemize}
\noindent Therefore, we have found, by means of new analytical representations, a
direct connection between nonlinear parabolic equations and linear hyperbolic
equations.\break  From the probabilistic point\vadjust{\goodbreak} of view this connection could be expected
because in both the cases we have generalizations of the diffusion equation
leading to a finite speed propagation.

%s2.4 #&#
\subsection{Nonlinear time-fractional diffusive equations admitting Barenblatt-type solutions}
\label{sec2.4}

There is a wide literature about the probabilistic interpretation\index{probabilistic interpretation} of linear
space and time-fractional diffusive equations (see, e.g.,
\cite{me1,me2,me3,or} and the references therein). On the other hand,
a probabilistic approach to time-fractional nonlinear diffusive-type equations
is still completly missing. %In recent papers
Recently, the existence and uniqueness
of compactly supported solutions for time-fractional porous medium equations\index{Porous Medium Equation (PME)}
has been investigated (see, e.g., \cite{pol}). However, up to our knowledge,
it is not possible to find an explicit form of the Barenblatt-type solution.
On the other hand, time-fractional diffusive equations are actracting an
increasing interest in the literature. Explicit Barenblatt-type solutions
for nonlinear time-fractional equations can play a relevant role for
future studies in this context. We here consider a new family of nonlinear
time-fractional diffusive equations admitting a Barenblatt-type solution
of the form \eqref{eq:gen}.

Let us consider the following nonlinear diffusive equation
%e20 #&#
\begin{equation}
\label{fbe}
\frac{1}{t^{2\nu }}\frac{\partial ^{\nu }u}{\partial t^{\nu }}+
\frac{1}{t^{\nu }}\frac{\partial ^{2} u}{\partial x^{2}}+\left (
\frac{\partial u}{\partial x}\right )^{2} = 0,
\end{equation}
where $\frac{\partial ^{\nu }}{\partial t^{\nu }}$ denotes the Riemann--Liouville
derivative
\begin{equation*}
\frac{\partial ^{\nu }f(x,t)}{\partial t^{\nu }} =
\frac{1}{\Gamma (1-\nu )} \frac{\partial }{\partial t}\int _{0}^{t} (t-s)^{-
\nu }f(x,s) \mathrm{d}s, \quad \nu \in (0,1),
\end{equation*}
and $f(x,\cdot )$ is a suitable well-behaved function (see
\cite{kilbas} for details about the functional setting).

We start our analysis from a simple ansatz:  equation \eqref{fbe} admits
a solution in the form
%e21 #&#
\begin{equation}
\label{solfe}
u(x,t) = \frac{C_{1}}{t^{\nu }}-C_{2} \frac{x^{2}}{t^{3\nu }}, \quad
\nu \in \left (0,1/3\right )\setminus \{1/4\} , \ (x,t)\in \mathbb{R}
\times \mathbb{R}^{+},
\end{equation}
where $C_{1}$ and $C_{2}$ are real constants that we are going to find.
We now directly check the correctness of this conjecture. We first recall
that, for $\nu >0$ and $\beta >-1$,
%e22 #&#
\begin{equation}
\label{pr}
\frac{\partial ^{\nu }}{\partial t^{\nu }}t^{\beta }=
\frac{\Gamma (1+\beta )}{\Gamma (1+\beta -\nu )}t^{\beta -\nu }.
\end{equation}
Then, by substituting \eqref{solfe} in \eqref{fbe} and by applying
\eqref{pr}, we have
%e23 #&#
\begin{equation}
\frac{\Gamma (1-\nu )}{\Gamma (1-2\nu )}\frac{C_{1}}{t^{4\nu }}-
\frac{\Gamma (1-3\nu )}{\Gamma (1-4\nu )}
\frac{C_{2} \ x^{2}}{t^{6\nu }}- \frac{2C_{2}}{t^{4\nu }}+
\frac{4C_{2}^{2} x^{2}}{t^{6\nu }}=0
\end{equation}
and balancing similar terms we have that
%e24 #&#
\begin{equation}
\begin{cases}
\displaystyle\frac{\Gamma (1-\nu )}{\Gamma (1-2\nu )}C_{1} - 2 C_{2} =0,
\\
\displaystyle 4C_{2}^{2}-\frac{\Gamma (1-3\nu )}{\Gamma (1-4\nu )}C_{2}=0.
\end{cases}
\end{equation}
Therefore
%e25 #&#
\begin{equation}
\label{cost}
\begin{cases}
\displaystyle C_{1} = \frac{1}{2}\frac{\Gamma (1-3\nu )}{\Gamma (1-4\nu )}
\frac{\Gamma (1-2\nu )}{\Gamma (1-\nu )},
\\
\displaystyle C_{2} = \frac{1}{4}\frac{\Gamma (1-3\nu )}{\Gamma (1-4\nu )}.
\end{cases}
\end{equation}

Observe that the constraint on the real order of derivation
$\nu \in \left (0,\frac{1}{3}\right )$ in \eqref{solfe} is due to the application
of \eqref{pr}.

Moreover, $\nu \neq 1/4$ because of the coefficients appearing in the solution
(depending on the Euler Gamma functions). We can conclude that  equation
\eqref{fbe} admits a solution of the form \eqref{solfe} if
$\nu \in (0,1/3)\setminus \{1/4\}$. Therefore, in particular, under the
same constraints on $\nu $, we can say that the time-fractional equation
\eqref{fbe} admits a solution of the form
%e26 #&#
\begin{equation}
\label{sf}
u(x,t) = \frac{C_{1}}{t^{\nu }}\left (1-\frac{C_{2}}{C_{1}}
\frac{x^{2}}{t^{2\nu }}\right )_{+},
\end{equation}
where $C_{1}$ and $C_{2}$ are given by \eqref{cost}.

Obviously, the solution \eqref{sf} is not normalized but it is a Barenblatt-type
solution belonging to the general family considered in this paper. A systematic
study of equation \eqref{fbe} should be an object of further investigation,
both from the physical and mathematical points of view. We conjecture that
this is a source-type solution of the nonlinear time-fractional equation
\eqref{fbe}; nevertheless a full rigorous analysis should be developed,
but this is beyond the aims of this paper. The fractional equation
\eqref{fbe} can be viewed as a hybrid between a diffusive equation with
singular time-dependent coefficients (in some way similar to the EPD equation\index{EPD equation})
and a nonlinear time-fractional porous medium type equation.\index{Porous Medium Equation (PME)}

%s3 #&#
\section{Main results}
\label{sec3}

Let us start with our first result concerning the case $d=1$.

%p1 #&#
\begin{proposition}\label{prop1}
For $d=1$, the density function\index{density function} \eqref{eq:gen} can be written  as
%e27 #&#
\begin{align}
\label{eq:dens2}
u(x,t)=\mathbf{E}_{V}\left [
\frac{\delta (x-Vt^{\alpha })+\delta (x+Vt^{\alpha })}{2}\right ],
\end{align}
where $\mathbf{E}_{V}[\cdot ]$ stands for the mean value w.r.t.
$V\stackrel{\text{(law)}}{=}cY^{1/\beta }$, where $Y\sim \mathrm{Beta} (
\frac{1}{\beta },\gamma +1)$.
\end{proposition}
\begin{proof}
Let $d=1$. We have
\begin{align*}
\hat{u}(\xi ,t)&=\mathcal{F}u(\xi ,t)
\\
&=Ct^{-\alpha }\int _{-ct^{\alpha }}^{ct^{\alpha }}e^{i\xi x}\left (1-
\left (\frac{|x|}{ct^{\alpha }}\right )^{\beta }\right )^{\gamma
}\mathrm{d}x
\\
&=2 Ct^{-\alpha }\int _{0}^{ct^{\alpha }}\cos (\xi x)\left (1-\left (
\frac{x}{ct^{\alpha }}\right )^{\beta }\right )^{\gamma }\mathrm{d}x
\\
&=\frac{\beta /c}{\mathrm{Beta}(\frac{1}{\beta },\gamma +1)}\int _{0}^{c}
\cos (\xi vt^{\alpha })\left (1- (v/c)^{\beta }\right )^{\gamma
}\mathrm{d}v
\\
&=\mathbf{E}_{V}\left [ \cos (\xi Vt^{\alpha })\right ]
\\
&=\mathbf{E}_{V}\left [
\frac{e^{i\xi Vt^{\alpha }}+e^{-i\xi Vt^{\alpha }}}{2}\right ].
\end{align*}
Hence, by Fubini's theorem we immediately obtain
\begin{equation*}
u(x,t)=\mathcal{F}^{-1}\hat{u}(\xi ,t)=\mathbf{E}_{V}\left [
\frac{\delta (x-Vt^{\alpha })+\delta (x+Vt^{\alpha })}{2}\right ],
\end{equation*}
where in the last step we used the  result
$\mathcal{F}^{-1} e^{\pm ia\xi }=\delta (x\mp a), a\in \mathbb{R}$ (see,
e.g., \cite{gr}).
\end{proof}

Based on Proposition \ref{prop1} we argue the following random model. Let
$D$ be a random variable uniformly distributed on $\{-1,1\}$, which is
independent from $V$. We deal with the  stochastic process
$X:=(X(t))_{t\geq 0}$, where
\begin{equation*}
X(t):=D \,V \,t^{\alpha
}\end{equation*}
represents the position, at time $t>0$, of a particle starting from the
origin of the real line, which initially chooses with the same probability
to move leftward or rightward and performs a random displacement of length
equal to $V t^{\alpha }$. Therefore $V$ represents the random velocity of
the particle which is initially fixed with the probability law\index{probability law} $f_{V}$.
%c1 #&#
\begin{corollary}
At time $t>0$
\begin{equation*}
P(X(t)\in \mathrm{d}x)=\mathbf{E}_{V}\left [
\frac{\delta (x-Vt^{\alpha })+\delta (x+Vt^{\alpha })}{2}\right ]
\mathrm{d}x
\end{equation*}
and the cumulative distribution function of $X(t)$ is given by
%e28 #&#
\begin{align}
\label{eq:cf}
{\mathbf{F}}(x)=\frac{1}{2}\left [1+\text{sgn}(x)
\frac{\mathrm{Beta}(|x|/ct^{\alpha };\frac{1}{\beta },\gamma +1)}{\mathrm{Beta}(\frac{1}{\beta },\gamma +1)}
\right ],\quad x\in [-ct^{\alpha },ct^{\alpha }].
\end{align}
\end{corollary}
\begin{proof}
Given $V=v$, one has that
\begin{align*}
P(X(t)\in A|V=v)&=\int _{A}\nu _{v}(x,t)\mathrm{d}x,\quad A\in
\mathcal{B}(\mathbb{R}),
\end{align*}
where
\begin{equation*}
\nu _{v}(x,t):=
\frac{\delta (x-vt^{\alpha })+\delta (x+vt^{\alpha })}{2}
\end{equation*}
represents the singular probability measure of $X_{1}$ (w.r.t. the Lebesgue
measure $\mathrm{d}x$), and then
\begin{align*}
\mu _{t}(\mathrm{d}x)&=P(X(t)\in \mathrm{d}x)=\mathbf{E}_{V}\left [
\nu _{V}(x,t)\right ]\mathrm{d}x.
\end{align*}

Some simple calculations and the result \eqref{eq: ball} lead to
\eqref{eq:cf}.
\end{proof}

The following corollary highlights the link between $u(x,t)$ and a model
of nonlinear wave propagation.
%c2 #&#
\begin{corollary}\label{c3}
Let $v>0$. Then
\begin{equation*}
\nu _{v}(x,t):=
\frac{\delta (x-vt^{\alpha })+\delta (x+vt^{\alpha })}{2}
\end{equation*}
is the fundamental solution\index{fundamental solution} to the  hyperbolic EPD-type partial
differential equation
%e29 #&#
\begin{equation}
\label{eq:wavesol}
\frac{\partial ^{2} u}{\partial t^{2}}+\frac{(1-\alpha )}{t}
\frac{\partial u}{\partial t}=v^{2}\alpha ^{2} t^{2\alpha -2}
\frac{\partial ^{2} u}{\partial x^{2}}
\end{equation}
subject to the initial conditions
$u(x,0)=\delta (x), \frac{\partial u}{\partial t}(x,t)\big |_{t=0}=0$.
\end{corollary}
\begin{proof}
We observe that $\frac{\delta (x-vs)+\delta (x+vs)}{2}, v>0,s>0$, is the
fundamental solution\index{fundamental solution} to the wave equation
\begin{equation*}
\frac{\partial ^{2} u}{\partial s^{2}}=v^{2}
\frac{\partial ^{2} u}{\partial x^{2}},
\end{equation*}
subject to the initial condition
$u(x,0)=\delta (x), \frac{\partial u}{\partial s}(x,s)\big |_{s=0}=0$. Therefore,
the change of variable $s=t^{\alpha }$ and direct calculations permit to
prove \eqref{eq:wavesol}.
\end{proof}

%r3.1 #&#
\begin{remark}
The Erd\'{e}lyi--Kober fractional integral is defined as (see, e.g.,
\cite{gianni})
%e30 #&#
\begin{equation}
\label{eq:eki}
I^{\zeta , \mu }_{\eta }f(x) =
\frac{\eta \ x^{-\eta (\mu +\zeta )}}{\Gamma (\mu )}\int _{0}^{x}
\tau ^{\eta (\zeta +1)-1}(x^{\eta }-\tau ^{\eta })^{\mu -1}f(\tau )
\mathrm{d}\tau ,
\end{equation}
where $\mu >0,\eta >0$ and $\zeta \in \mathbb{R}$. Evidently, for
$\zeta =0,\eta =1$, \eqref{eq:eki} reduces to the Riemann--Liouville integral
with a power weight.

From Proposition \ref{prop1} it is easy to check that the Fourier transform
$\hat{u}(\xi ,t)$ and $u(x,t)$ are Erd\'{e}lyi--Kober integrals of the cosine
function
$f_{\alpha }(v;\xi ,t):=\cos \left (\xi v t^{\alpha } \right )$ and
$g_{\alpha }(v;x,t):=
\frac{\delta (x-vt^{\alpha } )+\delta (x+vt^{\alpha } )}{2}$, respectively;
i.e.,
%e31 #&#
\begin{align}
\hat{u}(\xi ,t)& =
\frac{\Gamma (\frac{1}{\beta }+\gamma +1)}{\Gamma (\frac{1}{\beta })}I_{\beta }^{
\frac{1}{\beta }-1,\gamma +1}f_{\alpha }(c;\xi ,t)\end{align}
and
%e32 #&#
\begin{align}
u(x,t) &=
\frac{\Gamma (\frac{1}{\beta }+\gamma +1)}{\Gamma (\frac{1}{\beta })}I_{\beta }^{
\frac{1}{\beta }-1,\gamma +1}g_{\alpha }(c;x,t).
\end{align}
A recent interesting probabilistic interpretation\index{probabilistic interpretation} of the Erd\'{e}lyi--Kober integral
is also discussed in \cite{Tarasov}.

\end{remark}
%r3.2 #&#
\begin{remark}
The (centred) Wigner law is defined by the  probability distribution
%e33 #&#
\begin{equation}
\label{eq:wigner}
\mathfrak{m} (x,t)=\frac{1}{2\pi t}\sqrt{4t-x^{2}},\quad |x|\leq 2
\sqrt{t}.
\end{equation}
A simple calculation proves that the even moments are given by (scaled) Catalan
numbers, that is,
\begin{equation*}
\int _{-2\sqrt{t}}^{2\sqrt{t}}x^{2m}\mathfrak{m} ( x,t)\mathrm{d}x=C_{m}
t^{m}, \quad m\in \mathbb{N},
\end{equation*}
with $C_{m}=\frac{1}{m+1}\binom{2m}{m}$. The probability law\index{probability law}
\eqref{eq:wigner} is the density function\index{density function} of the free Brownian motion\index{Brownian motion}
$S:=(S_{t})_{t\geq 0}$, i.e. for $0\leq t_{1}<t_{2}<\infty $, the law of
$S_{t_{2}}-S_{t_{1}}$ is given by $ \mathfrak{m} (x,t_{2}-t_{1})$ and
$\mathbf{E}(S_{t_{2}}-S_{t_{1}})=0$,
$\mathbf{E}(S_{t_{2}}-S_{t_{1}})^{2}=t_{2}-t_{1}$. For a, detailed introduction
to the free probability and free Brownian motion\index{Brownian motion} the reader can consult,
for instance, \cite{voi,voi2} and \cite{nica}.

By setting $d=1,\alpha =1/2,\beta =2, \gamma =1/2$ and $c=2$, the function
\eqref{eq:gen} coincides with $ \mathfrak{m} (x,t)$. Furthermore, we observe
that $\mathfrak{m}:=\mathfrak{m} ( x,t)$ is equal to the time-rescaled, with
$t=s^{2}$, solution to the EPD equation \eqref{epd}, given by
%e34 #&#
\begin{equation}
\frac{\partial ^{2} u}{\partial s^{2}}+\frac{1}{s}
\frac{\partial u}{\partial s} = 4
\frac{ \partial ^{2} u}{\partial s^{2}}, \quad   s>0.
\end{equation}
Therefore some simple calculations allow to deduce
\begin{equation*}
\frac{\partial ^{2} \mathfrak{m}}{\partial t^{2}}+\frac{1}{2t}
\frac{\partial \mathfrak{m}}{\partial t}=\frac{1}{t}
\frac{\partial ^{2} \mathfrak{m}}{\partial x^{2}}.
\end{equation*}
\end{remark}

The study of the Barenblatt-type solutions for $d\geq 2$ leads to the following
alternative representations of \eqref{eq:gen}.

%p2 #&#
\begin{proposition}\label{propcf}
For $d\geq 2$, the probability density functions\index{probability density function} \eqref{eq:gen} have the
 representation
%e35 #&#
\begin{align}
\label{eq:dfd1}
u(x,t)=
\frac{\mathrm{Beta}(\frac{1}{\beta }(\frac{d}{2}+1),\gamma +1)}{\sigma (\mathbb{S}^{d-1})(ct^{\alpha }\|x\|)^{\frac{d}{2}-1}\mathrm{Beta}(\frac{d}{\beta },\gamma +1)}{
\mathbf{E}}_{ Z}\left [\delta (\|x\|- Z t^{\alpha })\right ]
\end{align}
where $\mathbf{E}_{Z}[\cdot ]$ stands for the mean value w.r.t.
$Z\stackrel{\text{(law)}}{=}cY_{1}^{1/\beta }$ where $Y_{1}\sim \mathrm{Beta}(
\frac{1}{\beta }(\frac{d}{2}+1),\gamma +1)$.
\end{proposition}
\begin{proof}
Let $d\geq 2$. Let $\sigma $ be the measure on $\mathbb{S}^{d-1}$. We recall
that (see (2.12), \xch{p.~690}{pag.690}, \cite{dgo}),
%e36 #&#
\begin{equation}
\label{eq:int}
\int _{\mathbb{S}^{d-1}}e^{i\rho \xi \cdot \theta } \mathrm{d}\sigma ({
\theta })=(2\pi )^{d/2}
\frac{J_{\frac{d}{2}-1}(\rho \|\xi \|)}{(\rho \|\xi \|)^{\frac{d}{2}-1}}.
\end{equation}
One has that
%e37 #&#
\begin{align}
\label{eq:cf1}
\hat{u}(\xi ,t)&=\mathcal{F}u(\xi ,t)
\nonumber\\
&=C\int _{0}^{ct^{\alpha }}\rho ^{d-1} t^{-\alpha d}\left (1-\left (
\frac{\rho }{ct^{\alpha }}\right )^{\beta }\right )^{\gamma }
\mathrm{d}\rho \int _{\mathbb{S}^{d-1}}e^{i\rho \xi \cdot \theta }
\mathrm{d}\sigma (\theta)
\nonumber
\\
&=\frac{t^{-\alpha d}C(2\pi )^{d/2}}{\|\xi \|^{\frac{d}{2}-1}}\int _{0}^{ct^{
\alpha }}\rho ^{\frac{d}{2}} \left (1-\left (\frac{\rho }{ct^{\alpha }}
\right )^{\beta }\right )^{\gamma }J_{\frac{d}{2}-1}(\rho \|\xi \|)
\mathrm{d}\rho
\nonumber
\\
&=\left (\frac{2}{ct^{\alpha }\|\xi \|}\right )^{\frac{d}{2}-1}
\frac{\Gamma (d/2)\beta /c}{\mathrm{Beta}(\frac{d}{\beta },\gamma +1)}
\int _{0}^{c} (z/c)^{\frac{d}{2}}(1-(z/c)^{\beta })^{\gamma }J_{
\frac{d}{2}-1}\left (t^{\alpha }z\|\xi \|\right )\mathrm{d}z
\nonumber
\\
&=\left (\frac{2}{ct^{\alpha }\|\xi \|}\right )^{\frac{d}{2}-1}
\frac{\Gamma (d/2)\mathrm{Beta}(\frac{1}{\beta }(\frac{d}{2}+1),\gamma +1)}{\mathrm{Beta}(\frac{d}{\beta },\gamma +1)}
{\mathbf{E}}_{ Z}\left [J_{\frac{d}{2}-1}\left (Zt^{\alpha } \|\xi \|\right )
\right ],
\end{align}
where $Z\stackrel{\text{(law)}}{=}cY^{1/\beta }$ is the random variable
with the density function\index{density function} given by
\begin{equation*}
f_{Z}(z)=
\frac{\beta /c}{\mathrm{Beta}(\frac{1}{\beta }(\frac{d}{2}+1),\gamma +1)}(z/c)^{
\frac{d}{2}}(1-(z/c)^{\beta })^{\gamma }1_{0<z<c}.
\end{equation*}
Hence, by Fubini's theorem we obtain
\begin{align*}
u(x,t)&=\mathcal{F}^{-1}u(\xi , t)
\\
&=\frac{1}{(2\pi )^{d}}\int _{\mathbb{R}^{d}}e^{- i x\cdot \xi }u(\xi ,
t) \mathrm{d}\xi
\\
&=(\text{by passing to spherical coordinates})
\\
&=\frac{1}{(2\pi )^{d}}\left (\frac{2}{ct^{\alpha }}\right )^{\frac{d}{2}-1}
\frac{\Gamma (d/2)\beta /c}{\mathrm{Beta}(\frac{d}{\beta },\gamma +1)}\\
&\quad \times \int _{0}^{c} (z/c)^{\frac{d}{2}}(1-(z/c)^{\beta })^{
\gamma }\mathrm{d}z
\\
&\quad\times\int _{0}^{\infty }\rho ^{d/2} J_{\frac{d}{2}-1}\left (zt^{\alpha }\rho \right )\left (\int _{\mathbb{S}^{d-1}}e^{-i \rho x\cdot
\theta } \mathrm{d}\sigma (\theta ) \right )\mathrm{d}\rho
\\
&=(\text{by explotinig the result}\, \text{\eqref{eq:int}})
\\
&=\frac{1}{2\pi ^{d/2}}
\frac{\Gamma (d/2)}{(ct^{\alpha }\|x\|)^{\frac{d}{2}-1}}
\frac{\beta /c}{\mathrm{Beta}(\frac{d}{\beta },\gamma +1)}
\\
&\quad \times \int _{0}^{c} (z/c)^{\frac{d}{2}}(1-(z/c)^{\beta })^{
\gamma }\mathrm{d}z\int _{0}^{\infty }\rho J_{\frac{d}{2}-1}\left (\rho t^{
\alpha } u\right )J_{\frac{d}{2}-1 }(\rho \|x\|)\mathrm{d}\rho
\\
&=\frac{1}{2\pi ^{d/2}}
\frac{\Gamma (d/2)}{(ct^{\alpha }\|x\|)^{\frac{d}{2}-1}}
\frac{\beta /c}{\mathrm{Beta}(\frac{d}{\beta },\gamma +1)}\\
&\quad \times\int _{0}^{c} (z/c)^{
\frac{d}{2}}(1-(z/c)^{\beta })^{\gamma }\delta (\|x\|- zt^{\alpha })
\mathrm{d}z
\\
&=(\text{by Proposition 2, in \cite{kell}})
\\
&=\frac{1}{2\pi ^{d/2}}
\frac{\Gamma (d/2)}{(ct^{\alpha }\|x\|)^{\frac{d}{2}-1}}
\frac{\mathrm{Beta}(\frac{1}{\beta }(\frac{d}{2}+1),\gamma +1)}{\mathrm{Beta}(\frac{d}{\beta },\gamma +1)}{
\mathbf{E}}_{ Z}\left [\delta (\|x\|- Z t^{\alpha })\right ]
\end{align*}
which coincides with \eqref{eq:dfd1}.
\end{proof}

%r3.3 #&#
\begin{remark}

For $d=2$, the representation \eqref{eq:dfd1} is particularly simple and
reads as
\begin{align*}
u(x,t)=\frac{1}{2\pi }{\mathbf{E}}_{ Z}\left [\delta (\|x\|- Z t^{\alpha })
\right ].
\end{align*}
\end{remark}

Let
\begin{equation*}
g_{uw}(x,t):=\mathcal{F}^{-1} e^{i \|\xi \| u w t}=
\frac{1}{(2\pi )^{d}}\int _{\mathbb{R}^{d}} e^{-i x\cdot \xi } e^{i \|
\xi \| u w t}\mathrm{d}\xi _{d} .
\end{equation*}
We are able to provide an alternative representation of $u(x,t)$ with respect
to \eqref{eq:dfd1}.
%p3 #&#
\begin{proposition}
For $d\geq 2$, the law \eqref{eq:gen} can be rewritten as
%e38 #&#
\begin{align}
\label{eq:dfd2}
u(x,t)=\mathbf{E}_{\mathfrak{U} \,\mathfrak{W}}\left [
\frac{g_{\mathfrak{U}\,\mathfrak{W}}(x,t^{\alpha })+g_{\mathfrak{U}\,\mathfrak{W}}(x,-t^{\alpha })}{2}
\right ]
\end{align}
where $\mathfrak{U}\stackrel{\text{(law)}}{=}cY_{2}^{1/\beta }$, with
$Y_{2}\sim \mathrm{Beta}(\frac{d}{\beta },\gamma +1)$, and $\mathfrak{W}$ admits the density
function\index{density function} given by
$f_{\mathfrak{W}}(w)=\frac{2}{\mathrm{Beta}(\frac{1}{2},\frac{d-1}{2})}(1-w^{2})_{+}^{
\frac{d-1}{2}-1}$. Moreover $\mathfrak{U}$ and $\mathfrak{W}$ are independent.
\end{proposition}
\begin{proof}

From \eqref{eq:cf1}, one has that
%e39 #&#
\begin{align}
\label{eq:cfd1}
\hat{u}(\xi ,t)&=\left (\frac{2}{ct^{\alpha }\|\xi \|}\right )^{\frac{d}{2}-1}
\frac{\Gamma (d/2)\beta /c}{\mathrm{Beta}(\frac{d}{\beta },\gamma +1)}
\int _{0}^{c} (z/c)^{\frac{d}{2}}(1-(z/c)^{\beta })^{\gamma }J_{
\frac{d}{2}-1}\left (t^{\alpha }z\|\xi \|\right )\mathrm{d}z
\nonumber
\\
&=
\frac{2\beta /c}{\mathrm{Beta}(\frac{d}{\beta },\gamma +1)\mathrm{Beta}(\frac{1}{2},\frac{d-1}{2})}
\int _{0}^{c}(z/c)^{d-1}(1-(z/c)^{\beta })^{\gamma } \mathrm{d}z
\nonumber
\\
&\quad \times \int _{0}^{1} (1-w^{2})^{\frac{d-1}{2}-1}\cos (\|\xi \| zw
t^{\alpha })\mathrm{d}w
\nonumber
\\
&=\mathbf{E}_{\mathfrak{U}\,\mathfrak{W}}\left [\cos ( \|\xi \|
\mathfrak{U}\,\mathfrak{W}\,t^{\alpha })\right ]
\end{align}
where we have used
%e40 #&#
\begin{equation}
\label{eq:poisson}
J_{\mu }(z)=\frac{(z/2)^{\mu }}{\sqrt{\pi }\Gamma (\mu +\frac{1}{2})}\int _{-1}^{+1}(1-w^{2})^{
\mu -\frac{1}{2}}\cos (z w)\mathrm{d}w
\end{equation}
valid for $\mu >-\frac{1}{2},z\in \mathbb{R}$. Therefore, from
\eqref{eq:cfd1} we get \eqref{eq:dfd2}. Indeed,
\begin{align*}
u(x,t)&=\mathcal{F}^{-1}\hat{u}(\xi , t)
\\
&=\frac{1}{(2\pi )^{d}}
\frac{2\beta /c}{\mathrm{Beta}(\frac{d}{\beta },\gamma +1)\mathrm{Beta}(\frac{1}{2},\frac{d-1}{2})}
\int _{0}^{c}(z/c)^{d-1}(1-(z/c)^{\beta })^{\gamma } \mathrm{d}z
\nonumber
\\
&\quad \times \int _{0}^{1} (1-w^{2})^{\frac{d-1}{2}-1}\mathrm{d}w
\int _{\mathbb{R}^{d}} e^{-i x\cdot \xi }\left [
\frac{e^{i \|\xi \| zw t^{\alpha }}+e^{- i\|\xi \| z w t^{\alpha }}}{2}
\right ]\mathrm{d}\xi
\nonumber
\end{align*}
which concludes the proof.
\end{proof}

%r3.4 #&#
\begin{remark}

We observe that:
\begin{itemize}
\item For $d=2$, the density function\index{density function} $f_{\mathfrak{W}}$ becomes the probability
law\index{probability law} of the square root of $T$, where
\begin{equation*}
T:=\text{meas}\{t\in [0,1]: B(t)>0\},
\end{equation*}
and $(B(t))_{t\geq 0}$ represents the standard one-dimensional Brownian
motion. $T$ leads to the well-known arcsin law of the Wiener process which
is given by
\begin{equation*}
\frac{1}{\pi \sqrt{w(1-w)}}1_{0<w<1}.
\end{equation*}
It is easy to prove that
$\sqrt{T}\stackrel{\text{(law)}}{=} \mathfrak{W}$.
\item For $d=3$, the random variable $\mathfrak{W}, t>0$, is uniformly distributed
in $(0,1)$.
\item For $d\geq 4$, the density function\index{density function} $f_{\mathfrak{W}}$ represents a
Wigner $(d-2)$-sphere law.
\end{itemize}

\end{remark}

%r3.5 #&#
\begin{remark}
It is simple to prove by direct calculations that the mean squared displacement
goes like $t^{2\alpha }$, recovering the entire range of behaviours from
sub-diffusion ($\alpha <1/2$) to super-diffusion ($\alpha >1/2$); i.e.,
\begin{equation*}
\int _{\mathbb{R}^{d}} \|x\|^{2} u(x,t)\mathrm{d}x=
\frac{\Gamma (\frac{1}{\beta }(d+2))\Gamma (\frac{d}{\beta }+\gamma +1)}{\Gamma (\frac{d}{\beta })\Gamma (\frac{1}{\beta }(d+2)+\gamma +1)}c^{2}t^{2
\alpha }.
\end{equation*}
\end{remark}

%\begin{appendix}
%\end{appendix}

\begin{acknowledgement}
The authors wish to thank the referees for their comments which allowed
the improvement of the previous version of the manuscript.
%[title={Acknowledgments}]
\end{acknowledgement}

%\begin{funding}
%\gsponsor[id=,sponsor-id=]{}
%\gnumber[refid=]{}
%\end{funding}


\begin{thebibliography}{35}

%b1 ###
\bibitem{me2}
\begin{barticle}
\bauthor{\bsnm{Baeumer}, \binits{B.}},
\bauthor{\bsnm{Meerschaert}, \binits{M.M.}},
\bauthor{\bsnm{Nane}, \binits{E.}}:
\batitle{Space-time duality for fractional diffusion}.
\bjtitle{J. Appl. Probab.}
\bvolume{46}(\bissue{4}),
\bfpage{1100}--\blpage{1115}
(\byear{2009}).
\bid{doi={10.1239/jap/\\1261670691}, mr={2582709}}
\end{barticle}
%
\OrigBibText
Baeumer, B., Meerschaert, M. M., Nane, E. (2009). Space-time duality for
fractional diffusion, \emph{Journal of Applied Probability}, 46(4), 1100-1115.
\endOrigBibText
\bptok{structpyb}
\endbibitem

%b2 ###
\bibitem{ben}
\begin{barticle}
\bauthor{\bsnm{Benachour}, \binits{S.}},
\bauthor{\bsnm{Chassaing}, \binits{P.}},
\bauthor{\bsnm{Roynette}, \binits{B.}},
\bauthor{\bsnm{Vallois}, \binits{P.}}:
\batitle{Processus associ\'{e}s a l'\'{e}quation des milieux poreux}.
\bjtitle{Ann. Sc. Norm. Super. Pisa, Cl. Sci.}
\bvolume{4},
\bfpage{793}--\blpage{832}
(\byear{1996}).
\bid{mr={1469575}}
\end{barticle}
%
\OrigBibText
Benachour, S., Chassaing, P., Roynette, B., Vallois, P. (1996) Processus
associ\'{e}s a l'\'{e}quation des milieux poreux, \emph{Annali della Scuola
Superiore di Pisa}, \textbf{4}, 793-832.
\endOrigBibText
\bptok{structpyb}
\endbibitem

%b3 ###
\bibitem{biler0}
\begin{barticle}
\bauthor{\bsnm{Biler}, \binits{P.}},
\bauthor{\bsnm{Imbert}, \binits{C.}},
\bauthor{\bsnm{Karch}, \binits{G.}}:
\batitle{Barenblatt profiles for a non local porous medium equation}.
\bjtitle{C. R. Acad. Sci. Paris, Ser. I}
\bvolume{349},
\bfpage{641}--\blpage{645}
(\byear{2011}).
\bid{doi={\\10.1016/j.crma.2011.06.003}, mr={2817383}}
\end{barticle}
%
\OrigBibText
Biler, P., Imbert, C., Karch, G. (2011) Barenblatt profiles for a non local
porous medium equation, \emph{C.R. Acad. Sci. paris, Ser. I}, \textbf{349},
641-645.
\endOrigBibText
\bptok{structpyb}
\endbibitem

%b4 ###
\bibitem{biler}
\begin{barticle}
\bauthor{\bsnm{Biler}, \binits{P.}},
\bauthor{\bsnm{Imbert}, \binits{C.}},
\bauthor{\bsnm{Karch}, \binits{G.}}:
\batitle{The Nonlocal Porous Medium Equation: Barenblatt Profiles and Other Weak Solutions}.
\bjtitle{Arch. Ration. Mech. Anal.}
\bvolume{215},
\bfpage{497}--\blpage{529}
(\byear{2015}).
\bid{doi={10.1007/s00205-014-0786-1}, mr={3294409}}
\end{barticle}
%
\OrigBibText
Biler, P., Imbert, C., Karch, G. (2015) The Nonlocal Porous Medium Equation:
Barenblatt Profiles and Other Weak Solutions, \emph{Arch. Rational Mech.
Anal.}, \textbf{215}, 497-529.
\endOrigBibText
\bptok{structpyb}
\endbibitem

%b5 ###
\bibitem{br}
\begin{barticle}
\bauthor{\bsnm{Bresters}, \binits{D.W.}}:
\batitle{On the equation of Euler-Poisson-Darboux}.
\bjtitle{SIAM J. Math. Anal.}
\bvolume{4}(\bissue{1}),
\bfpage{31}--\blpage{41}
(\byear{1973}).
\bid{doi={10.1137/0504005}, mr={0324235}}
\end{barticle}
%
\OrigBibText
Bresters, D. W., On the equation of Euler-Poisson-Darboux, SIAM Journal
on Mathematical Analysis 4.1 (1973): 31-41.
\endOrigBibText
\bptok{structpyb}
\endbibitem

%b6 ###
\bibitem{me3}
\begin{barticle}
\bauthor{\bsnm{Chen}, \binits{Z.Q.}},
\bauthor{\bsnm{Meerschaert}, \binits{M.M.}},
\bauthor{\bsnm{Nane}, \binits{E.}}:
\batitle{Space-time fractional diffusion on bounded domains}.
\bjtitle{J. Math. Anal. Appl.}
\bvolume{393}(\bissue{2}),
\bfpage{479}--\blpage{488}
(\byear{2012}).
\bid{doi={\\10.1016/j.jmaa.2012.04.032}, mr={2921690}}
\end{barticle}
%
\OrigBibText
Chen, Z. Q., Meerschaert, M. M., Nane, E. (2012). Space-time fractional
diffusion on bounded domains, \emph{Journal of Mathematical Analysis and
Applications}, \textbf{393(2)}, 479-488.
\endOrigBibText
\bptok{structpyb}
\endbibitem

%b7 ###
\bibitem{dgo}
\begin{barticle}
\bauthor{\bparticle{De} \bsnm{Gregorio}, \binits{A.}},
\bauthor{\bsnm{Orsingher}, \binits{E.}}:
\batitle{Flying randomly in $\mathbb{R}^{d}$ with Dirichlet displacements}.
\bjtitle{Stoch. Process. Appl.}
\bvolume{122},
\bfpage{676}--\blpage{713}
(\byear{2012}).
\bid{doi={10.1016/\\j.spa.2011.10.009}, mr={2868936}}
\end{barticle}
%
\OrigBibText
De Gregorio, A., Orsingher, E. (2012) Flying randomly in
$\mathbb{R}^{d}$ with Dirichlet displacements, \emph{Stochastic Processes
and their Applications}, \textbf{122}, 676-713.
\endOrigBibText
\bptok{structpyb}
\endbibitem

%b8 ###
\bibitem{DG18}
\begin{barticle}
\bauthor{\bparticle{De} \bsnm{Gregorio}, \binits{A.}}:
\batitle{Stochastic models associated to a Nonlocal Porous Medium Equation}.
\bjtitle{Mod. Stoch. Theory Appl.}
\bvolume{5},
\bfpage{457}--\blpage{470}
(\byear{2018}).
\bid{doi={10.15559/\\18-vmsta112}, mr={3914725}}
\end{barticle}
%
\OrigBibText
De Gregorio, A. (2018) Stochastic models associated to a Nonlocal Porous
Medium Equation, \emph{Modern Stochastics: Theory and Applications}, \textbf{5},
457-470.
\endOrigBibText
\bptok{structpyb}
\endbibitem

%b9 ###
\bibitem{dgo2}
\begin{botherref}
\oauthor{\bparticle{De} \bsnm{Gregorio}, \binits{A.}},
\oauthor{\bsnm{Orsingher}, \binits{E.}}:
Random flights connecting Porous Medium and Euler-Poisson-Darboux equations
(2017).
\arxivurl{arXiv:1709.07663}, 20 pp.
\end{botherref}
%
\OrigBibText
De Gregorio, A., Orsingher, E. (2017) Random flights connecting Porous
Medium and Euler-Poisson-Darboux equations, arXiv:1709.07663, 20 pp.
\endOrigBibText
\bptok{structpyb}
\endbibitem

%b10 ###
\bibitem{erd}
\begin{barticle}
\bauthor{\bsnm{Erd\'{e}lyi}, \binits{A.}}:
\batitle{On the Euler-Poisson-Darboux equation}.
\bjtitle{J. Anal. Math.}
\bvolume{23}(\bissue{1}),
\bfpage{89}--\blpage{102}
(\byear{1970}).
\bid{doi={10.1007/BF02795492}, mr={0285807}}
\end{barticle}
%
\OrigBibText
A. Erd\'{e}lyi, On the Euler-Poisson-Darboux equation,
\emph{Journal d'Analyse Math\'{e}matique}, 23(1), 89--102, (1970)
\endOrigBibText
\bptok{structpyb}
\endbibitem

%b11 ###
\bibitem{feng}
\begin{barticle}
\bauthor{\bsnm{Feng}, \binits{S.}},
\bauthor{\bsnm{Iscoe}, \binits{I.}},
\bauthor{\bsnm{Sepp\"{a}l\"{a}inen}, \binits{T.}}:
\batitle{A microscopic mechanism for the porous medium equation}.
\bjtitle{Stoch. Process. Appl.}
\bvolume{66},
\bfpage{147}--\blpage{182}
(\byear{1997}).
\bid{doi={\\10.1016/S0304-4149(96)00121-4}, mr={1440397}}
\end{barticle}
%
\OrigBibText
Feng, S., Iscoe, I., Sepp\"{a}l\"{a}inen, T. (1997) A microscopic mechanism
for the porous medium equation, \emph{Stochastic Processes and their Applications},
\textbf{66}, 147-182.
\endOrigBibText
\bptok{structpyb}
\endbibitem

%b12 ###
\bibitem{garra3}
\begin{barticle}
\bauthor{\bsnm{Garra}, \binits{R.}},
\bauthor{\bsnm{Orsingher}, \binits{E.}}:
\batitle{Random Flights Related to the Euler-Poisson-Darboux Equation}.
\bjtitle{Markov Process. Relat. Fields}
\bvolume{22},
\bfpage{87}--\blpage{110}
(\byear{2016}).
\bid{mr={3523980}}
\end{barticle}
%
\OrigBibText
Garra, R., Orsingher, E. (2016) Random Flights Related to the Euler-Poisson-Darboux
Equation, \emph{Markov Processes and Related Fields}, \textbf{22}, 87-110.
\endOrigBibText
\bptok{structpyb}
\endbibitem

%b13 ###
\bibitem{getoor}
\begin{barticle}
\bauthor{\bsnm{Getoor}, \binits{R.K.}}:
\batitle{First passage times for symmetric stable processes in space}.
\bjtitle{Trans. Am. Math. Soc.}
\bvolume{101},
\bfpage{75}--\blpage{90}
(\byear{1961}).
\bid{doi={10.1090/S0002-9947-1961-0137148-5}, doi={10.2307/1993412}, mr={0137148}}
\end{barticle}
%
\OrigBibText
Getoor, R.K. (1961) First passage times for symmetric stable processes
in space, \emph{Trans. Amer. Math. Soc.}, \textbf{101}, 75-90.
\endOrigBibText
\bptok{structpyb}
\endbibitem

%b14 ###
\bibitem{gr}
\begin{bbook}
\bauthor{\bsnm{Gradshteyn}, \binits{I.S.}},
\bauthor{\bsnm{Ryzhik}, \binits{I.M.}}:
\bbtitle{Tables of Integrals, Series, and Products},
\bedition{4}th edn.
\bpublisher{Academic Press},
\blocation{New York}
(\byear{1980}).
\bid{mr={0669666}}
\end{bbook}
%
\OrigBibText
Gradshteyn, I.S., Ryzhik, I.M. (1980) \emph{Tables of Integrals, Series,
and Products}. Fourth edition. Academic Press, New York.
\endOrigBibText
\bptok{structpyb}
\endbibitem

%b15 ###
\bibitem{leo}
\begin{barticle}
\bauthor{\bsnm{Heyde}, \binits{C.C.}},
\bauthor{\bsnm{Leonenko}, \binits{N.N.}}:
\batitle{Student processes}.
\bjtitle{Adv. Appl. Probab.}
\bvolume{37},
\bfpage{342}--\blpage{365}
(\byear{2005}).
\bid{doi={10.1239/aap/1118858629}, mr={2144557}}
\end{barticle}
%
\OrigBibText
Heyde, C.C., Leonenko, N.N. (2005) Student processes, \emph{Advances in
Applied Probability}, \textbf{37}, 342-365.
\endOrigBibText
\bptok{structpyb}
\endbibitem

%b16 ###
\bibitem{inoue}
\begin{barticle}
\bauthor{\bsnm{Inoue}, \binits{M.}}:
\batitle{A Markov process associated with a porous medium equation}.
\bjtitle{Proc. Jpn. Acad.}
\bvolume{60},
\bissue{Ser. A},
\bfpage{157}--\blpage{160}
(\byear{1989}).
\bid{doi={10.3792/pjaa.60.157}, mr={0758056}}
\end{barticle}
%
\OrigBibText
Inoue, M. (1989) A Markov process associated with a porous medium equation,
\emph{Proc. Japan Acad.}, \textbf{60}, Ser. A, 157-160.
\endOrigBibText
\bptok{structpyb}
\endbibitem

%b17 ###
\bibitem{inoue2}
\begin{barticle}
\bauthor{\bsnm{Inoue}, \binits{M.}}:
\batitle{Construction of diffusion processes associated with a porous medium equation}.
\bjtitle{Hiroshima Math. J.}
\bvolume{19},
\bfpage{281}--\blpage{297}
(\byear{1989}).
\bid{doi={10.32917/hmj/\\1206129389}, mr={1027932}}
\end{barticle}
%
\OrigBibText
Inoue, M. (1989) Construction of diffusion processes associated with a
porous medium equation, \emph{Hiroshima Mathematical Journal}, \textbf{19},
281-297.
\endOrigBibText
\bptok{structpyb}
\endbibitem

%b18 ###
\bibitem{inoue3}
\begin{barticle}
\bauthor{\bsnm{Inoue}, \binits{M.}}:
\batitle{Derivation of a porous medium equation from many Markovian particles
and the propagation of chaos}.
\bjtitle{Hiroshima Math. J.}
\bvolume{21},
\bfpage{85}--\blpage{110}
(\byear{1991}).
\bid{doi={10.32917/hmj/1206128924}, mr={1091433}}
\end{barticle}
%
\OrigBibText
Inoue, M. (1991) Derivation of a porous medium equation from many Markovian
particles and the propagation of chaos, \emph{Hiroshima Mathematical Journal},
\textbf{21}, 85-110.
\endOrigBibText
\bptok{structpyb}
\endbibitem

%b19 ###
\bibitem{jou}
\begin{barticle}
\bauthor{\bsnm{Jourdain}, \binits{B.}}:
\batitle{Probabilistic approximation for a porous medium equation}.
\bjtitle{Stoch. Process. Appl.}
\bvolume{89},
\bfpage{81}--\blpage{99}
(\byear{2000}).
\bid{doi={10.1016/S0304-4149(00)00014-4}, mr={1775228}}
\end{barticle}
%
\OrigBibText
Jourdain, B. (2000) Probabilistic approximation for a porous medium equation,
\emph{Stochastic Processes and their Applications}, \textbf{89}, 81-99.
\endOrigBibText
\bptok{structpyb}
\endbibitem

%b20 ###
\bibitem{kam}
\begin{barticle}
\bauthor{\bsnm{Kamin}, \binits{S.}},
\bauthor{\bsnm{Vazquez}, \binits{J.L.}}:
\batitle{Fundamental solutions and asymptotic behaviour for the $p$-Laplacian equation}.
\bjtitle{Rev. Mat. Iberoam.}
\bvolume{4},
\bfpage{339}--\blpage{354}
(\byear{1988}).
\bid{doi={\\10.4171/RMI/77}, mr={1028745}}
\end{barticle}
%
\OrigBibText
Kamin, S., Vazquez, J.L. (1988) Fundamental solutions and asymptotic behaviour
for the $p$-Laplacian equation, \emph{Revista Matematica Iberoamericana},
\textbf{4}, 339-354.
\endOrigBibText
\bptok{structpyb}
\endbibitem

%b21 ###
\bibitem{kell}
\begin{barticle}
\bauthor{\bsnm{Kellendong}, \binits{J.}},
\bauthor{\bsnm{Richard}, \binits{S.}}:
\batitle{Weber-Schafheitlin-type integrals with exponent 1}.
\bjtitle{Integral Transforms Spec. Funct.}
\bvolume{20},
\bfpage{147}--\blpage{153}
(\byear{2009}).
\bid{doi={10.1080/\\10652460802321485}, mr={2492212}}
\end{barticle}
%
\OrigBibText
Kellendong, J., Richard, S. (2009) Weber-Schafheitlin-type integrals with
exponent 1, \emph{Integral Transforms and Special Functions}, \textbf{20},
147-153.
\endOrigBibText
\bptok{structpyb}
\endbibitem

%b22 ###
\bibitem{kilbas}
\begin{bchapter}
\bauthor{\bsnm{Kilbas}, \binits{A.A.}},
\bauthor{\bsnm{Srivastava}, \binits{H.M.}},
\bauthor{\bsnm{Trujillo}, \binits{J.J.}}:
\bctitle{Theory and applications of fractional differential equations}.
In: \bbtitle{Vol.},
vol.~\bseriesno{204}.
\bpublisher{Elsevier Science Limited}
(\byear{2006}).
\bid{mr={2218073}}
\end{bchapter}
%
\OrigBibText
Kilbas, A.A., Srivastava, H.M., Trujillo, J.J. (2006). \emph{Theory and
applications of fractional differential equations}, (Vol. 204). Elsevier
Science Limited.
\endOrigBibText
\bptok{structpyb}
\endbibitem

%b23 ###
\bibitem{lee}
\begin{barticle}
\bauthor{\bsnm{Lee}, \binits{K.}},
\bauthor{\bsnm{Petrosyan}, \binits{A.}},
\bauthor{\bsnm{Vazquez}, \binits{J.L.}}:
\batitle{Large-time geometric properties of solutions of the evolution $p$-Laplacian equation}.
\bjtitle{J. Differ. Equ.}
\bvolume{229},
\bfpage{389}--\blpage{411}
(\byear{2006}).
\bid{doi={10.1016/j.jde.2005.07.028}, mr={2263560}}
\end{barticle}
%
\OrigBibText
Lee, K., Petrosyan, A., Vazquez, J.L. (2006) Large-time geometric properties
of solutions of the evolution $p$-Laplacian equation, \emph{Journal of Differential
Equations}, \textbf{229}, 389-411.
\endOrigBibText
\bptok{structpyb}
\endbibitem

%b24 ###
\bibitem{me1}
\begin{barticle}
\bauthor{\bsnm{Meerschaert}, \binits{M.M.}},
\bauthor{\bsnm{Nane}, \binits{E.}},
\bauthor{\bsnm{Vellaisamy}, \binits{P.}}:
\batitle{Fractional Cauchy problems on bounded domains}.
\bjtitle{Ann. Probab.}
\bvolume{37}(\bissue{3}),
\bfpage{979}--\blpage{1007}
(\byear{2009}).
\bid{doi={10.1214/\\08-AOP426}, mr={2537547}}
\end{barticle}
%
\OrigBibText
Meerschaert, M. M., Nane, E., Vellaisamy, P. (2009). Fractional Cauchy
problems on bounded domains, \emph{The Annals of Probability}, \textbf{37(3)},
979-1007.
\endOrigBibText
\bptok{structpyb}
\endbibitem

%b25 ###
\bibitem{mendes}
\begin{barticle}
\bauthor{\bsnm{Mendes}, \binits{R.S.}},
\bauthor{\bsnm{Lenzi}, \binits{E.K.}},
\bauthor{\bsnm{Malacarne}, \binits{L.C.}},
\bauthor{\bsnm{Picoli}, \binits{S.}},
\bauthor{\bsnm{Jauregui}, \binits{M.}}:
\batitle{Random walks associated with nonlinear Fokker-Planck equations}.
\bjtitle{Entropy}
\bvolume{19}
(\byear{2017}),
Paper No. \bnumber{155},
\bcomment{11 pp.}
\bid{doi={10.3390/e19040155}, mr={3653180}}
\end{barticle}
%
\OrigBibText
Mendes, R.S., Lenzi, E.K., Malacarne, L.C., Picoli, S., Jauregui, M. (2017)
Random walks associated with nonlinear Fokker-Planck equations, \emph{Entropy},
\textbf{19}, Paper No. 155, 11 pp.
\endOrigBibText
\bptok{structpyb}
\endbibitem

%b26 ###
\bibitem{nica}
\begin{bbook}
\bauthor{\bsnm{Nica}, \binits{A.}},
\bauthor{\bsnm{Speicher}, \binits{R.}}:
\bbtitle{Lectures on the Combinatorics of Free Probability}.
\bsertitle{London Mathematical Society Lecture Note Series},
vol.~\bseriesno{335}.
\bpublisher{Cambridge Univ. Press},
\blocation{Cambridge}
(\byear{2006}).
\bid{doi={10.1017/CBO9780511735127}, mr={2266879}}
\end{bbook}
%
\OrigBibText
Nica, A., Speicher, R. (2006) \emph{Lectures on the Combinatorics of Free
Probability}. London Mathematical Society Lecture Note Series 335. Cambridge
Univ. Press, Cambridge.
\endOrigBibText
\bptok{structpyb}
\endbibitem

%b27 ###
\bibitem{or}
\begin{barticle}
\bauthor{\bsnm{Orsingher}, \binits{E.}},
\bauthor{\bsnm{Beghin}, \binits{L.}}:
\batitle{Fractional diffusion equations and processes with randomly varying time}.
\bjtitle{Ann. Probab.}
\bvolume{37}(\bissue{1}),
\bfpage{206}--\blpage{249}
(\byear{2009}).
\bid{doi={\\10.1214/08-AOP401}, mr={2489164}}
\end{barticle}
%
\OrigBibText
Orsingher, E., Beghin, L. (2009). Fractional diffusion equations and processes
with randomly varying time, \emph{The Annals of Probability}, \textbf{37(1)},
206-249.
\endOrigBibText
\bptok{structpyb}
\endbibitem

%b28 ###
\bibitem{gianni}
\begin{barticle}
\bauthor{\bsnm{Pagnini}, \binits{G.}}:
\batitle{Erd\'{e}lyi-Kober fractional diffusion}.
\bjtitle{Fract. Calc. Appl. Anal.}
\bvolume{15},
\bfpage{117}--\blpage{127}
(\byear{2012}).
\bid{doi={10.2478/s13540-012-0008-1}, mr={2872114}}
\end{barticle}
%
\OrigBibText
Pagnini, G. (2012). Erd\'{e}lyi-Kober fractional diffusion, \emph{Fractional
Calculus and Applied Analysis}, \textbf{15}, 117-127.
\endOrigBibText
\bptok{structpyb}
\endbibitem

%b29 ###
\bibitem{phil}
\begin{barticle}
\bauthor{\bsnm{Philipowski}, \binits{R.}}:
\batitle{Interacting diffusions approximating the porous medium equation and propagation of chaos}.
\bjtitle{Stoch. Process. Appl.}
\bvolume{117},
\bfpage{526}--\blpage{538}
(\byear{2007}).
\bid{doi={10.1016/j.spa.2006.09.003}, mr={2305385}}
\end{barticle}
%
\OrigBibText
Philipowski, R. (2007) Interacting diffusions approximating the porous
medium equation and propagation of chaos, \emph{Stochastic Processes and
their Applications}, \textbf{117}, 526-538.
\endOrigBibText
\bptok{structpyb}
\endbibitem

%b30 ###
\bibitem{pol}
\begin{botherref}
\oauthor{\bsnm{Plociniczak}, \binits{L.}},
\oauthor{\bsnm{\'{S}witala}, \binits{M.}}:
Compactly supported solution of the time-fractional porous medium
equation on the half-line
(2018).
\arxivurl{arXiv:1803.03016}.
\bid{doi={10.1137/18M1192561}, mr={3924620}}
\end{botherref}
%
\OrigBibText
Plociniczak, L., \'{S}witala, M. (2018). Compactly supported solution of
the time-fractional porous medium equation on the half-line. arXiv preprint
arXiv:1803.03016.
\endOrigBibText
\bptok{structpyb}
\endbibitem

%b31 ###
\bibitem{Rosen}
\begin{barticle}
\bauthor{\bsnm{Rosencrans}, \binits{S.I.}}:
\batitle{Diffusion Transforms}.
\bjtitle{J. Differ. Equ.}
\bvolume{13},
\bfpage{457}--\blpage{467}
(\byear{1973}).
\bid{doi={10.1016/0022-0396(73)90004-1}, mr={0331533}}
\end{barticle}
%
\OrigBibText
Rosencrans, S. I. (1973). Diffusion Transforms,
\emph{Journal of Differential Equations}, 13, 457--467.
\endOrigBibText
\bptok{structpyb}
\endbibitem

%b32 ###
\bibitem{Tarasov}
\begin{barticle}
\bauthor{\bsnm{Tarasov}, \binits{V.E.}},
\bauthor{\bsnm{Tarasova}, \binits{S.}}:
\batitle{Probabilistic Interpretation of Kober Fractional Integral of Non-Integer Order}.
\bjtitle{Prog. Fract. Differ. Appl.}
\bvolume{5}(\bissue{1}),
\bfpage{1}--\blpage{5}
(\byear{2019}).
\bid{doi={\\10.18576/pfda/050101}}
\end{barticle}
%
\OrigBibText
Tarasov, V. E., Tarasova, S. (2019). Probabilistic Interpretation of Kober
Fractional Integral of Non-Integer Order, \emph{Progr. Fract. Differ. Appl.},
\textbf{5(1)}, 1-5.
\endOrigBibText
\bptok{structpyb}
\endbibitem

%b33 ###
\bibitem{vazquez}
\begin{bbook}
\bauthor{\bsnm{Vazquez}, \binits{J.L.}}:
\bbtitle{The Porous Medium Equation. Mathematical Theory}.
\bsertitle{Oxford Math. Monogr.}
\bpublisher{Oxford Univ. Press},
\blocation{Oxford}
(\byear{2007}).
\bid{mr={2286292}}
\end{bbook}
%
\OrigBibText
Vazquez, J.L. (2007) \emph{The Porous Medium Equation. Mathematical Theory}.
Oxford Math. Monogr., Oxford Univ. Press, Oxford.
\endOrigBibText
\bptok{structpyb}
\endbibitem

%b34 ###
\bibitem{voi}
\begin{barticle}
\bauthor{\bsnm{Voiculescu}, \binits{D.}}:
\batitle{Limit laws for random matrices and free products}.
\bjtitle{Invent. Math.}
\bvolume{104},
\bfpage{201}--\blpage{220}
(\byear{1991}).
\bid{doi={10.1007/BF01245072}, mr={1094052}}
\end{barticle}
%
\OrigBibText
Voiculescu, D. (1991). Limit laws for random matrices and free products,
\emph{Invent. Math.}, \textbf{104}, 201-220.
\endOrigBibText
\bptok{structpyb}
\endbibitem

%b35 ###
\bibitem{voi2}
\begin{bbook}
\bauthor{\bsnm{Voiculescu}, \binits{D.V.}},
\bauthor{\bsnm{Dykema}, \binits{K.J.}},
\bauthor{\bsnm{Nica}, \binits{A.}}:
\bbtitle{Free Random Variables: A Noncommutative Probability Approach
to Free Products with Applications to Random Matrices, Operator
Algebras and Harmonic Analysis on Free Groups}.
\bsertitle{CRM Monograph Series},
vol.~\bseriesno{1}.
\bpublisher{Amer. Math. Soc.},
\blocation{Providence, RI}
(\byear{1992}).
\bid{mr={1217253}}
\end{bbook}
%
\OrigBibText
Voiculescu, D. V., Dykema, K. J. and Nica, A. (1992). \emph{Free Random
Variables: A Noncommutative Probability Approach to Free Products with
Applications to Random Matrices, Operator Algebras and Harmonic Analysis
on Free Groups.} CRM Monograph Series 1. Amer. Math. Soc., Providence,
RI.
\endOrigBibText
\bptok{structpyb}
\endbibitem

\end{thebibliography}
\end{document}